\documentclass[leqno]{article}

\usepackage{amsmath, amssymb, amsthm}
\usepackage{hyperref}
\usepackage{cleveref}
\usepackage{chngcntr}
\usepackage{thmtools}
\usepackage{thm-restate}
\usepackage{fancyhdr}
\usepackage{dsfont}
\usepackage[english]{babel}
\usepackage[utf8]{inputenc}
\usepackage{commath}
\usepackage{marvosym}
\usepackage{mathtools}

\usepackage{thmtools}
\usepackage{thm-restate}
\usepackage{soul}

\usepackage{tikz-cd}

\usepackage[margin=1.5in]{geometry}

\usepackage{comment}

\usepackage{enumitem}

\newtheorem{theorem}{Theorem}[section]
\newtheorem{corollary}[theorem]{Corollary}
\newtheorem{proposition}[theorem]{Proposition}

\newtheorem{lemma}[theorem]{Lemma}

\newtheorem*{theorem*}{Theorem}
\newtheorem*{corollary*}{Corollary}
\newtheorem*{lemma*}{Lemma}
\newtheorem*{proposition*}{Proposition}

\theoremstyle{definition}
\newtheorem{definition}[theorem]{Definition}
\newtheorem{remark}[theorem]{Remark}
\newtheorem{example}[theorem]{Example}

\newtheorem{thmx}{Theorem}[section]

\DeclareMathOperator{\length}{\ell}
\DeclareMathOperator{\pdim}{pd}
\DeclareMathOperator{\tor}{Tor}
\DeclareMathOperator{\ext}{Ext}

\DeclareMathOperator{\ann}{ann}
\DeclareMathOperator{\depth}{depth}
\DeclareMathOperator{\Hom}{Hom}
\DeclareMathOperator{\supp}{Supp}
\DeclareMathOperator{\height}{ht}
\DeclareMathOperator{\spec}{Spec}

\DeclareMathOperator{\grade}{grade}

\DeclareMathOperator{\cx}{cx}
\DeclareMathOperator{\px}{px}

\title{On the Ext Analog of the Euler Characteristic}
\author{Benjamin Katz\footnote{bkatz2@huskers.unl.edu}\text{ } and Andrew J. Soto Levins\footnote{ansotole@ttu.edu}}
\date{}

\begin{document}
\maketitle
\begin{abstract}
The Euler form is an Ext analog of the Euler characteristic, and in this paper we study the Euler form and give some applications. The first being a question of Jorgensen, which bounds the projective dimension of a module over a complete intersection by using the vanishing of self extensions. Our second application uses the Euler form to yield a new result involving the vanishing of the higher Herbrand difference. Along the way we translate some of our results to the graded setting.
\end{abstract}

\section{Introduction}
Serre's work on intersection multiplicities is a fundamental topic in commutative algebra. Let $(R,m,k)$ be a Noetherian local ring and let $M$ and $N$ be finitely generated modules. We write $\length(-)$ for the length of a module. For a fixed integer $j$, if $\length(\tor_{i}^{R}(M,N))<\infty$ for $i\geq j$ and if $\tor_{i}^{R}(M,N)=0$ for $i\gg 0$, we define the \textit{partial Euler Characteristic} $\chi_{j}^{R}(M,N)$ to be
\[\chi_{j}^{R}(M,N)=\sum_{i\geq j}(-1)^{i-j}\length(\tor_{i}^{R}(M,N)).\]
We write $\chi^{R}(M,N)$ for $\chi_{0}^{R}(M,N)$. Recall that $R$ is \textit{unramified} if $R$ is equicharacteristic or is mixed characteristic and char($k$)$\not\in m^{2}$. Work of Hochster \cite{hochster2} and of Lichtenbaum \cite{lichtenbaum} proves the following.
\begin{theorem}[Hochster and Lichtenbaum] \label{hochsterlichtenbaumtheorem} Let $R$ be an unramified regular local ring and let $M$ and $N$ be finitely generated modules. Let $j\geq 1$ be an integer and assume that $\length(\tor_{i}^{R}(M,N))<\infty$ for $i\geq j$. Then $\chi_{j}^{R}(M,N)=0$ if and only if $\tor_{i}^{R}(M,N)=0$ for all $i\geq j$.
\end{theorem}
For a fixed integer $j$, if $\length(\ext_{R}^{i}(M,N))<\infty$ for $i\geq j$ and if $\ext_{R}^{i}(M,N)=0$ for $i\gg 0$, we define the \textit{partial Euler Form} $\xi_{j}^{R}(M,N)$ to be
\[\xi_{j}^{R}(M,N)=\sum_{i\geq j}(-1)^{i-j}\length(\ext_{R}^{i}(M,N)).\]
We write $\xi^{R}(M,N)$ for $\xi_{0}^{R}(M,N)$. The integer $\xi^{R}(M,N)$ is called the \textit{Euler Form} and it was first studied in \cite{mori2}, also see \cite{jeanchan, mori, mori3}.\newline

The work in this paper began with the following question: are the partial Euler forms related to the vanishing of Ext in the same way the partial Euler characteristics are related to the vanishing of Tor? The first theorem of this paper is the following. The \textit{grade of $M$ against $N$ over $R$}, written $\grade_{R}(M,N)$, is the index of the first nonzero Ext module of $M$ against $N$ (see Section \ref{maintheoremsection} for a more detailed definition).
\begin{thmx}[Lemma \ref{signsofpartialeulerforms} and Theorem \ref{mainresult}] \label{maintheorem} Let $R$ be an unramified regular local ring, let $M$ and $N$ be finitely generated modules with $\length(M\otimes_{R}N)<\infty$ and $\dim{M}+\dim{N}<\dim{R}$, and assume $j$ is an integer with 
\[1\leq j\leq \grade_{R}M.\]
Then $\xi_{j}^{R}(M,N)\geq 0$ and the following are equivalent.
\begin{enumerate}
\item $\xi_{j}^{R}(M,N)=0$.
\item $\grade_{R}(M,N)\geq j$.
\item $\chi_{\pdim_{R}M-j+1}^{R}(M,N)=0$.
\end{enumerate}
\end{thmx}
This theorem tells us that the partial Euler forms can detect the vanishing of Ext, but not in the same way that the partial Euler characteristics detect the vanishing of Tor. From the equivalence of (1) and (2) and the definition of $\grade_R(M,N)$ we have $\xi_{j}^{R}(M,N)=0$ if and only if $\ext_{R}^{n}(M,N)=0$ for $n<j$, whereas $\chi_{j}^{R}(M,N)=0$ is equivalent to $\tor_{n}^{R}(M,N)=0$ for $n\geq j$. In Example \ref{nawajexample} we show that the partial Euler forms can be negative, something that cannot happen with the partial Euler characteristics over an unramified regular local ring.\newline

The rest of the paper is devoted to using the Euler form to prove results related to the vanishing of Ext. In Theorem \ref{moreimprovedjorgensenquestiontheorem} we make progress on a question of Jorgensen \cite[Question 2.7]{jorgensen}, which asks must $\ext_{R}^{n}(M,M)$ be nonzero for $0\leq n\leq\pdim_{R}M$ for a nonzero finitely generated module $M$ of finite projective dimension over a local complete intersection. In Theorem \ref{gradedjeanchan} we prove a graded version of a formula of Chan \cite[Theorem 5]{jeanchan}. Chan showed that if $M$ and $N$ are finitely generated modules of finite projective dimension over a local complete intersection $R$ with $\length(M\otimes_{R}N)<\infty$, then
\[\chi^{R}(M,N)=(-1)^{\grade_{R}M}\xi^{R}(M,N).\]
We end the paper by giving a new case for when the higher Herbrand difference vanishes (see Theorem \ref{vanishingofextlemmausingherbranddifference}).\newline

We would like to thank Lars Christensen, Nawaj KC, and Mark Walker for the many helpful conversations throughout this project. We would also like to thank Justin Lyle and Alexandra Seceleanu for answering our Macaulay2 questions.

\section{The Partial Euler Forms} \label{maintheoremsection}
The goal of this section is to define and study the partial Euler forms. In the proposition below we summarize the facts about the grade of a pair of modules $\grade_{R}(M,N)$ that we will use. We denote $\grade_{R}(M,R)$ by $\grade_{R}{M}$. The following is \cite[Proposition 1.2.10 and Definition 1.2.11]{brunsherzog}.
\begin{proposition} \label{gradeproposition} Let $R$ be a Noetherian ring and let $M$ and $N$ be finitely generated modules. Then
\begin{align*}
\grade_{R}(M,N) = \inf\{n|\ext_{R}^{n}(M,N)\neq 0\} &= \text{ longest $N$ regular sequence in $\ann{M}$} \\
&= \inf\{\depth{N_{p}}|p\in\supp{M}\cap\supp{N}\}.
\end{align*}
\end{proposition}

The following summarizes what is known about the Euler form.
\begin{proposition} \label{xizerosummary} Let $R$ be a local complete intersection and let $M$ and $N$ be finitely generated modules with $\pdim_{R}M<\infty$, $\pdim_{R}N<\infty$, and $\length(M\otimes_{R}N)<\infty$. Then the following hold.
\begin{enumerate}
\item If $\dim{M}+\dim{N}<\dim{R}$, then $\xi^{R}(M,N)=0$.
\item If $R$ is regular and unramified and $\dim{M}+\dim{N}=\dim{R}$, then 
\[(-1)^{\grade_R{M}}\xi^{R}(M,N)>0.\]
\end{enumerate}
\end{proposition}

\begin{proof}
In \cite[Theorem 5]{jeanchan} Chan proves
\begin{equation} \label{marksaidso}
\chi^{R}(M,N)=(-1)^{\grade_{R}{M}}\xi^{R}(M,N).
\end{equation}
If $\dim{M}+\dim{N}<\dim{R}$, then $\chi^{R}(M,N)=0$ by Serre's Vanishing Theorem (see \cite{giletsoule, roberts2}). Serre's Positivity Theorem (see \cite{serre}) shows that $\chi^{R}(M,N)>0$ if $R$ is regular and unramified and $\dim{M}+\dim{N}=\dim{R}$. This and (\ref{marksaidso}) proves 2.
\end{proof}

\begin{definition} Let $R$ be a ring and let $M$ and $N$ be finitely generated modules. Let $j$ be an integer and assume $\length(\ext_{R}^{i}(M,N))<\infty$ for $i\leq j$. Define $\overline{\xi}_{j}^{R}(M,N)$ to be
\[\overline{\xi}_{j}^{R}(M,N)=\sum_{i=0}^{j}(-1)^{i}\length(\ext_{R}^{j-i}(M,N)).\]
\end{definition}

\begin{lemma} \label{superjeanchanformula} Let $R$ be a local complete intersection and let $M$ and $N$ be finitely generated modules with $\pdim_{R}M<\infty$, $\pdim_{R}N<\infty$, and $\length(M\otimes_{R}N)<\infty$. If $j$ is an integer, then
\[(-1)^{\grade_{R}{M}+j}\xi_{j}^{R}(M,N)=\chi^{R}(M,N)+(-1)^{\grade_{R}{M}+j}\overline{\xi}_{j-1}^{R}(M,N).\]
In particular, if $\chi^{R}(M,N)=0$, then $\xi_{j}^{R}(M,N)=\overline{\xi}_{j-1}^{R}(M,N)$.
\end{lemma}

\begin{proof}
Letting $g=\grade_{R}{M}$, note that \cite[Theorem 5]{jeanchan} gives the first equality below
\begin{align*}
\chi^{R}(M,N) = (-1)^{g}\xi^{R}(M,N) &= (-1)^{g}[(-1)^{j-1}\overline{\xi}_{j-1}^{R}(M,N)+(-1)^{j}\xi_{j}^{R}(M,N)] \\
&= (-1)^{g+j-1}\overline{\xi}_{j-1}^{R}(M,N)+(-1)^{g+j}\xi_{j}^{R}(M,N).\qedhere
\end{align*}
\end{proof}

Part 1 of the following gives part of Theorem \ref{maintheorem}.
\begin{lemma} \label{signsofpartialeulerforms} Let $R$ be an unramified regular local ring and let $M$ and $N$ be finitely generated modules with $\length(M\otimes_{R}N)<\infty$. If $j$ is an integer with $1\leq j\leq\grade_{R}{M}$, then the following hold. 
\begin{enumerate}
\item If $\dim{M}+\dim{N}<\dim{R}$, then $\xi_{j}^{R}(M,N)\geq 0$, and $\xi_{j}^{R}(M,N)=0$ if and only if $\grade_{R}(M,N)\geq j$.
\item If $\dim{M}+\dim{N}=\dim{R}$ and $\grade_{R}{M}+j$ is even, then $\xi_{j}^{R}(M,N)>0$.
\item If $\dim{M}+\dim{N}=\dim{R}$, $\grade_{R}{M}+j$ is odd, and $\ext_{R}^{j-1}(M,N)=0$, then $\xi_{j}^{R}(M,N)<0$.
\end{enumerate}
\end{lemma}

\begin{proof}
1. Since $\dim{M}+\dim{N}<\dim{R}$ implies $\chi^{R}(M,N)=0$ by Serre's Vanishing Theorem (see \cite{serre}), we have $\xi_{j}^{R}(M,N)=\overline{\xi}_{j-1}^{R}(M,N)$ by Lemma \ref{superjeanchanformula}. Using \cite[Lemma 2.6]{soto}, if $j\geq 2$, then $\overline{\xi}_{j-1}^{R}(M,N)\geq 0$, and $\overline{\xi}_{j-1}^{R}(M,N)=0$ if and only if $\ext_{R}^{i}(M,N)=0$ for $i\leq j-1$. The case $j=1$ is clear.\newline

2. Note that $\overline{\xi}_{j-1}^{R}(M,N)\geq 0$ for $2\leq j\leq \grade_{R}{M}$ by \cite[Lemma 2.6]{soto} and we always have $\overline{\xi}_{0}^{R}(M,N)\geq 0$. Lemma \ref{superjeanchanformula} gives the equality below
\[\xi_{j}^{R}(M,N)=\chi^{R}(M,N)+\overline{\xi}_{j-1}^{R}(M,N)\geq \chi^{R}(M,N)>0,\]
and Serre's Positivity Theorem (see \cite{serre}) gives the second inequality.\newline

3. By Serre's Positivity Theorem we have $\chi^{R}(M,N)>0$, and so by Lemma \ref{superjeanchanformula} we have $\overline{\xi}_{j-1}^{R}(M,N)>\xi_{j}^{R}(M,N)$, which implies it is enough to show $\overline{\xi}_{j-1}^{R}(M,N)=0$. Since $\ext_{R}^{j-1}(M,N)=0$ and since $j-1\leq\grade_{R}{M}$, $\ext_{R}^{i}(M,N)=0$ for $i\leq j-1$ by \cite[Theorem 2.9]{soto}, and so $\overline{\xi}_{j-1}^{R}(M,N)=0$.
\end{proof}

Before proving the rest of Theorem \ref{maintheorem}, we first need a lemma.
\begin{lemma} \label{lemmaformainresult} Let $(R,m,k)$ be an unramified regular local ring, let $M$ and $N$ be finitely generated modules with $\length(M\otimes_{R}N)<\infty$, and assume $j$ is an integer with $1\leq j\leq \pdim_{R}M$. Then $\chi_{j}^{R}(M,N)=0$ if and only if $\pdim_{R}M-j+1\leq \grade_{R}(M,N)$.
\end{lemma}

\begin{proof}
Let
\[q_{R}(M,N) = \sup\{n|\tor_{n}^{R}(M,N)\neq 0\},\]
the index of the last nonzero Tor of $M$ against $N$. By Theorem \ref{hochsterlichtenbaumtheorem}, $\chi_{j}^{R}(M,N)=0$ if and only if $j-1\geq q_{R}(M,N)$. Since 
\[q_{R}(M,N)=\depth{R}-\depth{M}-\depth{N}\]
by \cite[Theorem on page 110]{serre}, $j-1\geq q_{R}(M,N)$ if and only if $j-1\geq \pdim_{R}M-\depth{N}$. Since $\length(M\otimes_{R}N)<\infty$,
\[\grade_{R}(M,N)=\inf\{\depth{N_{p}}|p\in\{m\}\}=\depth{N},\]
and so $j-1\geq \pdim_{R}M-\depth{N}$ if and only if $j-1\geq \pdim_{R}M-\grade_{R}(M,N)$.
\end{proof}

The following makes up the other half of Theorem \ref{maintheorem} from the introduction. 
\begin{theorem} \label{mainresult} Let $R$ be an unramified regular local ring, let $M$ and $N$ be finitely generated modules with $\length(M\otimes_{R}N)<\infty$ and $\dim{M}+\dim{N}<\dim{R}$, and assume $j$ is an integer with 
\[1\leq j\leq \grade_{R}M.\]
Then the following are equivalent.
\begin{enumerate}
\item $\xi_{j}^{R}(M,N)=0$.
\item $\grade_{R}(M,N)\geq j$.
\item $\chi_{\pdim_{R}M-j+1}^{R}(M,N)=0$.
\end{enumerate}
\end{theorem}

\begin{proof}
The equivalence of (1) and (2) is by Lemma \ref{signsofpartialeulerforms}, and the equivalence of (2) and (3) follows from Lemma \ref{lemmaformainresult}.
\end{proof}

We end this section with three examples. By Lemma \ref{signsofpartialeulerforms}(3), if $\ext_{R}^{j-1}(M,N)=0$ then $\xi_{j}^{R}(M,N)<0$. The example below shows that the converse is false. 
\begin{example} \label{nawajexample} Let $R=\mathbb{Q}[x,y,z]$, $M=R/(y^{2},z^{2})$, and $N=R/(x)(x,y,z)$. Note that $\length(M\otimes_{R}N)<\infty$,
\[\dim{M}+\dim{N}=1+2=3=\dim{R},\]
and
\[\grade_{R}{M}=\dim{R}-\dim{M}=3-1=2.\]
Letting $j=1$, we see that $\grade_{R}{M}+j=2+1=3$ is odd. By a computation in Macaulay2 \cite{macaulay2} we get $\xi_{1}^{R}(M,N)=-3<0$, but $\ext_{R}^{0}(M,N)\neq 0$.
\end{example}

By Lemma \ref{signsofpartialeulerforms}(2), if $\dim{M}+\dim{N}=\dim{R}$ and $\grade_{R}{M}+j$ even then $\xi_{j}^{R}(M,N)>0$. The following example shows that if $\grade_{R}{M}+j$ is odd instead of even, $\xi_{j}^{R}(M,N)$ can be zero.
\begin{example} \label{nawajexampletwo} Let $R=\mathbb{Q}[x,y,z]$, $M=R/(y,z)$, and $N=R/(x)(x,y,z)$. Note that $\length(M\otimes_{R}N)<\infty$,
\[\dim{M}+\dim{N}=1+2=3=\dim{R},\]
and
\[\grade_{R}{M}=\dim{R}-\dim{M}=3-1=2.\]
Letting $j=1$, we see that $\grade_{R}{M}+j=2+1=3$ is odd. By a computation in Macaulay2 \cite{macaulay2} we get $\xi_{1}^{R}(M,N)=0$, and $\ext_{R}^{0}(M,N)\neq 0$.
\end{example}

Assume $\dim{M}+\dim{N}=\dim{R}$ and $\grade_{R}{M}+j$ odd. Example \ref{nawajexample} shows that $\xi_{j}^{R}(M,N)$ can be negative and Example \ref{nawajexampletwo} shows that $\xi_{j}^{R}(M,N)$ can be zero. The next example shows that $\xi_{j}^{R}(M,N)$ can be positive.
\begin{example} Let $R=\mathbb{Q}[x,y,z]$, $M=R/(y,z)$, and $N=R/(x)(x^{2},y,z)$. Note that $\length(M\otimes_{R}N)<\infty$,
\[\dim{M}+\dim{N}=1+2=3=\dim{R},\]
and
\[\grade_{R}{M}=\dim{R}-\dim{M}=3-1=2.\]
Letting $j=1$, we see that $\grade_{R}{M}+j=2+1=3$ is odd. By a computation in Macaulay2 \cite{macaulay2} we get $\xi_{1}^{R}(M,N)=1>0$, and $\ext_{R}^{0}(M,N)\neq 0$.
\end{example}

\section{On a Question of Jorgensen}
The goal of this section is to prove Theorem \ref{moreimprovedjorgensenquestiontheorem}. In \cite[Question 2.7]{jorgensen} Jorgensen asked if $\ext_{R}^{n}(M,M)$ must be nonzero for $0\leq n\leq\pdim_{R}M$ for a finitely generated module $M$ of finite projective dimension over a local complete intersection. The answer to this question is yes if $R$ is regular \cite{jothilingam} or an unramified hypersurface \cite[Proposition 5.4]{dao2}. Also, in \cite[Theorem 2.9]{soto} it was shown that if $M$ is a finitely generated module, not necessarily of finite projective dimension, over an unramified hypersurface $R$, then $\ext_{R}^{n}(M,M)\neq 0$ for $0\leq n\leq \grade_{R}M$. Note that we always have $\grade_{R}M\leq\pdim_{R}M$, and equality holds when $M$ is CM and has finite projective dimension.\newline

Before proving Theorem \ref{moreimprovedjorgensenquestiontheorem}, we first prove three lemmas.
\begin{lemma} \label{xizero} Let $(R,m,k)$ be a Noetherian local ring and let $M$ and $N$ be finitely generated modules with $\pdim_{R}M<\infty$ and $\length(N)<\infty$. If $\grade_{R}M>0$, then $\xi^{R}(M,N)=0$ and $\chi^{R}(M,N)=0$.
\end{lemma}

\begin{proof}
Note that $\grade_{R}M>0$ is equivalent to $\ann{M}$ containing a nonzerodivisor. We only prove $\xi^{R}(M,N)=0$. To show $\chi^{R}(M,N)=0$ is similar. Since $\ann{M}$ contains a nonzerodivisor, if $\length(N)=1$, then \cite[Theorem 19.8]{matsumura} gives us the fourth equality below:
\[\xi^{R}(M,N) = \xi^{R}(M,k) = \sum_{i=0}^{\pdim_{R}M}(-1)^{i}\length(\ext_{R}^{i}(M,k)) = \sum_{i=0}^{\pdim_{R}M}(-1)^{i}\beta_{i}(M) = 0,\]
where $\beta_{i}(M)$ is the $ith$ Betti number of $M$. An induction on $\length(N)$ using that $\xi^{R}(M,-)$ is additive on short exact sequences finishes the proof.
\end{proof}

\begin{lemma} \label{lemmaforjorgensensquestiontheorem} Let $R$ be a Noetherian local ring with $\dim{R}\leq 2$, and let $M$ and $N$ be nonzero finitely generated modules with $\pdim_{R}M<\infty$ and $\length(N)<\infty$. If $\grade_{R}M>0$, then $\ext_{R}^{n}(M,N)\neq 0$ and $\tor_{n}^{R}(M,N)\neq 0$ for $0\leq n\leq \pdim_{R}M$.
\end{lemma}

\begin{proof}
We prove the statement for Ext. The argument for the Tor statement is similar. Since $\Hom_{R}(M,N)\neq 0$ and $\ext_{R}^{\pdim_{R}M}(M,N)\neq 0$, we can assume $\pdim_{R}M=\dim{R}=2$. By Lemma \ref{xizero} we have $\xi^{R}(M,N)=0$. If $\ext_{R}^{1}(M,N)=0$, then we would have
\[\xi^{R}(M,N) = \length(\Hom_{R}(M,N)) + \length(\ext_{R}^{2}(M,N)) > 0,\]
a contradiction.
\end{proof}

\begin{lemma} \label{gradeisthesamelocallyatminimalprime} Let $R$ be a CM ring, let $M$ be a finitely generated module, and choose $p\in\supp{M}$ so that $\dim{R/p}=\dim{M}$. Then $\grade_{R}{M}=\grade_{R_{p}}M_{p}$.
\end{lemma}

\begin{proof}
Notice that 
\[\height{p} = \dim{R_{p}} = \depth{{R_{p}}} =\grade_{R_{p}}M_{p}.\]
Also, if $q\in\supp{M}$, then 
\[\dim{R}-\height{q}=\dim{R/q}\leq\dim{R/p}=\dim{R}-\height{p}\]
where the first and second equalities are because CM rings are catenary (\cite[Theorem 17.9]{matsumura}), and so $\height{p}\leq \height{q}$. This tells us that
\[\grade_{R}{M} = \inf\{\depth{R_{q}}|q\in\supp{M}\} = \inf\{\height{q}|q\in\supp{M}\} = \height{p}.\qedhere\]
\end{proof}

\begin{theorem} \label{moreimprovedjorgensenquestiontheorem} Let $R$ be a Noetherian ring and let $M$ be a nonzero finitely generated module with $\pdim_{R}M<\infty$. Assume either $R$ is local with equidimensional completion or $R$ is CM. If $\grade_{R}M\leq 2$, then the following hold
\begin{enumerate}
\item $\ext_{R}^{n}(M,M)\neq 0$ for $0\leq n\leq \grade_{R}M$.
\item $\tor_{n}^{R}(M,M)\neq 0$ for $0\leq n\leq \grade_{R}M$.
\end{enumerate}
\end{theorem}

\begin{proof}
First assume that $R$ is local with equidimensional completion. Since $\widehat{R}$ is a faithfully flat extension of $R$, we can assume $R$ is complete. Let $p$ be a minimal prime of $M$ so that $\dim{R/p}=\dim{M}$. Notice that
\begin{align*}
\grade_{R_{p}}M_{p} \leq \height{\ann_{R_{p}}M_{p}} \leq \dim{R_{p}}  = \height{p} &\leq \dim{R} - \dim{R/p} \\
&= \dim{R} - \dim{M} \\
&=\grade_{R}M.
\end{align*}
The last equality is by \cite[Theorem 3.6]{beder} since $R$ is equidimensional and $\pdim_{R}M<\infty$. Since $\grade_{R}M\leq\grade_{R_{p}}M_{p}$ holds in general, this shows $\grade_{R}M=\grade_{R_{p}}M_{p}$. Therefore
\[\dim{R_{p}} = \pdim_{R_{p}}M_{p} = \grade_{R_{p}}M_{p} = \grade_{R}M \leq 2,\]
and so the theorem follows from Lemma \ref{lemmaforjorgensensquestiontheorem}.\newline

Now assume $R$ is CM. By Lemma \ref{gradeisthesamelocallyatminimalprime} we have $\grade_{R}{M}=\grade_{R_{p}}M_{p}$ where $p$ is a minimal prime of $M$ with $\dim{R/p}=\dim{M}$, and so the result follows from the local case by localizing at $p$. 
\end{proof}

\begin{remark} Recall that given a homomorphism $f:Q\rightarrow R$ of rings, an $R$ module $M$ \textit{lifts} to a $Q$ modules $L$ if $L\otimes_{Q}R\cong M$ and $\tor_{n}^{Q}(L,R)=0$ for $n>0.$ Modules over regular local rings, modules of finite projective dimension over local unramified hypersurfaces, and modules that lift to a regular local ring are Tor rigid, and known cases of Jorgensen's question use this fact. In fact, in \cite[Theorem 2.1]{jorgensen}, Jorgensen showed that if $M$ is a finitely generated Tor rigid module of finite projective dimension over a Noetherian local ring $R$, then $\ext_{R}^{n}(M,M)\neq 0$ for $0\leq n\leq\pdim_{R}M$. In general modules of finite projective dimension are not Tor rigid; Heitmann gave the first example in \cite{heitmann}.
\end{remark}

\begin{corollary} Let $R$ be a CM ring, let $M$ be a nonzero finitely generated module with $\pdim_{R}M<\infty$, and assume there exists a nonzerodivisor $x\in\ann_{R}M$ so that $\ext_{R/x}^{i}(M,M)=0$ for some $i>\dim{R}$. If $\grade_{R}M=3$, then $\ext_{R}^{n}(M,M)\neq 0$ for $0\leq n\leq \grade_{R}M$.
\end{corollary}

\begin{proof}
By localizing at a minimal prime $p$ of $M$ with $\dim{R/p}=\dim{M}$, we can assume $R$ is local and $M$ has finite length; just note that $\grade_{R}M=\grade_{R_{p}}M_{p}$ by Lemma \ref{gradeisthesamelocallyatminimalprime}. Let $k$ be the residue field of $R$. Then the fact that $\pdim_{R}M<\infty$ and the following exact sequence implies that the Betti numbers of $M$ over $R/x$ are bounded
\[\begin{tikzcd}
	0 && {\ext_{R/x}^{0}(M,k)} && {\ext_{R}^{0}(M,k)} && {0 } \\
	&& {\ext_{R/x}^{1}(M,k)} && {\ext_{R}^{1}(M,k)} && {\ext_{R/x}^{0}(M,k)} \\
	&& {\ext_{R/x}^{2}(M,k)} && {\ext_{R}^{2}(M,k)} && {\ext_{R/x}^{1}(M,k)} && \cdots
	\arrow[from=1-1, to=1-3]
	\arrow[from=1-3, to=1-5]
	\arrow[from=1-5, to=1-7]
	\arrow[from=1-7, to=2-3]
	\arrow[from=2-3, to=2-5]
	\arrow[from=2-5, to=2-7]
	\arrow[from=2-7, to=3-3]
	\arrow[from=3-3, to=3-5]
	\arrow[from=3-5, to=3-7]
	\arrow[from=3-7, to=3-9]
\end{tikzcd}\]
(see \cite[Theorem 10.75]{rotman} for this long exact sequence). Since $M$ having finite complete intersection dimension over $R$ implies $M$ has finite complete intersection dimension over $R/x$ by \cite[ Proposition 1.12.3]{avramovgasharovpeeva}, we have $\pdim_{R/x}M<\infty$ by \cite[Proposition 2.5 and Proposition 4.16]{celikbasdao}. Since $\grade_{R/x}M=2$, Theorem \ref{moreimprovedjorgensenquestiontheorem} now gives $\ext_{R/x}^{n}(M,M)\neq 0$ for $n=0,1,2$. The corollary now follows from using the following exact sequence
\[\begin{tikzcd}
	0 && {\ext_{R/x}^{0}(M,M)} && {\ext_{R}^{0}(M,M)} && {0 } \\
	&& {\ext_{R/x}^{1}(M,M)} && {\ext_{R}^{1}(M,M)} && {\ext_{R/x}^{0}(M,M)} \\
	&& {\ext_{R/x}^{2}(M,M)} && {\ext_{R}^{2}(M,M)} && {\ext_{R/x}^{1}(M,M)} \\
	&& 0 && {\ext_{R}^{3}(M,M)} && {\ext_{R/x}^{2}(M,M)} \\
	&& 0
	\arrow[from=1-1, to=1-3]
	\arrow[from=1-3, to=1-5]
	\arrow[from=1-5, to=1-7]
	\arrow[from=1-7, to=2-3]
	\arrow[from=2-3, to=2-5]
	\arrow[from=2-5, to=2-7]
	\arrow[from=2-7, to=3-3]
	\arrow[from=3-3, to=3-5]
	\arrow[from=3-5, to=3-7]
	\arrow[from=3-7, to=4-3]
	\arrow[from=4-3, to=4-5]
	\arrow[from=4-5, to=4-7]
	\arrow[from=4-7, to=5-3]
\end{tikzcd}\]
(see \cite[Theorem 10.75]{rotman} for this long exact sequence).
\end{proof}

\section{On a Formula of Chan}
The goal of this section is to prove Theorem \ref{gradedjeanchan}. Before doing this, we need to set up some notation. The notation we use here is from \cite[Section 6.3]{roberts}. Let $R$ be a Noetherian $\mathbb{Z}$ graded ring with $R_{0}$ Artinian and let $M$ be a finitely generated graded module. For an integer $k$, the \textit{shift} of $M$ by $k$, denoted by $M[k]$, is the graded module where $M[k]_{n}=M_{n+k}$. The \textit{Hilbert polynomial} of $M$, denoted by $P_{M}$, is the polynomial where $P_{M}(n)=\length_{R_{0}}(\oplus_{k\leq n}M_{k})$ for $n\gg 0$. The degree of this polynomial is the dimension of $M$. The $ith$ derivative of $P_{M}$ is denoted by $P_{M}^{(i)}$.\newline

Let $F$ be a graded free resolution of $M$ with each $F_{i}$ finitely generated and $F_{i}=\oplus R[n_{ij}]$. Define $c_{k}$ to be
\[c_{k}=\frac{1}{k!} \sum_{i,j} (-1)^{i}n_{ij}^{k}.\]
The first part of our proof of the following theorem closely follows \cite[Theorem 6.3.1]{roberts}.
\begin{theorem} \label{gradedjeanchan} Let $R$ be a Noetherian $\mathbb{Z}$ graded ring with $R_{0}$ Artinian, and let $M$ and $N$ be finitely generated graded modules with $\length(M\otimes_{R}N)<\infty$ and $\pdim_{R}M<\infty$. Then
\[\chi^{R}(M,N)=(-1)^{\grade_{R}M}\xi^{R}(M,N).\]
\end{theorem}

\begin{proof}
Let $F$ be a graded free resolution of $M$ with each $F_{i}$ finitely generated and $F_{i}=\oplus R[n_{ij}]$ and let $k_{0}$ be the smallest integer $k$ so that $c_{k}\neq 0$. Then $k_{0}=\dim{R}-\dim{M}$ by the discussion after \cite[Theorem 6.3.1]{roberts}, and the discussion after the proof of \cite[Theorem 6.3.1]{roberts} gives the second equality below
\[\chi^{R}(M,N) = \sum_{i}(-1)^{i} P_{\tor_{i}^{R}(M,N)} = \sum_{k} c_{k}P_{N}^{(k)} = \sum_{k\geq k_{0}} c_{k}P_{N}^{(k)}.\]
Also, we have $\dim{M}+\dim{N}\leq \dim{R}$ by \cite[Theorem 6.3.2]{roberts}, and so $P_{N}^{(k)}=0$ for $k>k_{0}$. This implies $\sum_{k\geq k_{0}} c_{k}P_{N}^{(k)}=c_{k_{0}}P_{N}^{(k_{0})}$, and so
\begin{equation} \label{chiequalityforgradedstuff}
\chi^{R}(M,N) = c_{k_{0}}P_{N}^{(k_{0})}    
\end{equation}

Since 
\[\Hom_{R}(R[n],N)\cong N[-n],\]
$\Hom_{R}(F,N)$ is a complex where $\Hom_{R}(F,N)_{-i}\cong \oplus N[-n_{ij}]$. This gives the third equality below
\begin{align*}
\sum_{i} (-1)^{i} P_{\ext_{R}^{i}(M,N)}(n) = \sum_{i} (-1)^{i} P_{H_{-i}(\Hom_{R}(F,N))}(n) &= \sum_{i} (-1)^{i} P_{\Hom_{R}(F,N)_{-i}}(n) \\
&= \sum_{i,j} (-1)^{i} P_{N}(n-n_{ij}) \\
&= \sum_{i,j,k} (-1)^{i} P_{N}^{(k)}(n)\frac{(-n_{ij})^{k}}{k!} \\
&= \sum_{k} \frac{P_{N}^{(k)}(n)}{k!}(\sum_{i,j} (-1)^{i+k} n_{ij}^{k}) \\
&= \sum_{k} \frac{P_{N}^{(k)}(n)}{k!}(\sum_{i,j} (-1)^{i}(-1)^{k} n_{ij}^{k}) \\
&= \sum_{k} (-1)^{k}\frac{P_{N}^{(k)}(n)}{k!}(\sum_{i,j} (-1)^{i} n_{ij}^{k}) \\
&= \sum_{k } (-1)^{k} c_{k} P_{N}^{(k)}(n) \\
&= (-1)^{k_{0}} c_{k_{0}} P_{N}^{(k_{0})}(n).
\end{align*}
The additivity of Hilbert polynomials with respect to exact sequences gives the second equality and the discussion before \cite[Theorem 6.3.1]{roberts} gives the fourth equality. Therefore
\[(-1)^{k_{0}} \sum_{i} (-1)^{i} P_{\ext_{R}^{i}(M,N)}(n) = c_{k_{0}} P_{N}^{(k_{0})}(n)\]
for all $n$, giving the second equality below
\[(-1)^{k_{0}} \xi^{R}(M,N) = (-1)^{k_{0}} \sum_{i} (-1)^{i} P_{\ext_{R}^{i}(M,N)} = c_{k_{0}} P_{N}^{(k_{0})}.\]
Since $\grade_{R}M=\dim{R}-\dim{M}$ by \cite[Theorem 6.3.4]{roberts}, this and (\ref{chiequalityforgradedstuff}) completes the proof.
\end{proof}

\begin{remark} The formula in Theorem \ref{gradedjeanchan} was first proved in \cite[Theorem 5]{jeanchan} for local complete intersections when both modules have finite projective dimension. A similar graded result was proved in \cite[Corollary 3.11]{mori} when $R$ is positively graded, $R_{0}$ is a field, and without the finite projective dimension assumption.
\end{remark}

The following is the graded version of Theorem \ref{moreimprovedjorgensenquestiontheorem}. 
\begin{corollary} \label{gradedversion} Let $R$ be a Noetherian $\mathbb{Z}$ graded ring with $R_{0}$ Artinian and $\dim{R}\leq 2$, and let $M$ and $N$ be nonzero finite length graded modules with $\pdim_{R}M<\infty$. If $R$ has a unique graded maximal ideal, then $\ext_{R}^{n}(M,N)\neq 0$ for $0\leq n\leq \pdim_{R}M$.
\end{corollary}

\begin{proof}
Since
\[\pdim_{R}M = \sup\{\pdim_{R_{p}}M_{p}|p\in\supp{M}\} \leq \sup\{\dim{R_{p}}|p\in\spec{R}\} \leq \dim{R},\] 
we can assume $\pdim_{R}M=2$. This is because $\Hom_{R}(M,N)\neq 0$, and $\ext_{R}^{\pdim_{R}M}(M,N)\neq 0$ by the graded version of Nakayama's Lemma (see \cite[Exercise 1.5.24.a]{brunsherzog}). If $\ext_{R}^{1}(M,N)=0$, then we would have $\xi^{R}(M,N)>0$, giving
\[\chi^{R}(M,N) = (-1)^{\grade_{R}M}\xi^{R}(M,N) = \xi^{R}(M,N) >0.\]
The first equality is by Theorem \ref{gradedjeanchan}, and the second equality is because 
\[\grade_{R}M=\dim{R}-\dim{M}\]
by \cite[Theorem 6.3.4]{roberts}. But $\chi^{R}(M,N)=0$ by \cite[Theorem 6.3.2]{roberts}, giving a contradiction.
\end{proof}

\begin{example}
A ring that satisfies the assumptions in Corollary \ref{gradedversion} is the $\mathbb{Z}$ graded ring $R=k[x,y]/(x^{2})$ where $k$ is a field, the degree of $x$ is -1, the degree of $y$ is 1. Another example is a quotient of a standard graded polynomial ring over a field by a homogeneous ideal.
\end{example}

\section{The Euler Form and The Higher Herbrand Difference} \label{higherherbranddifferencesection}
The goal of this section is to give a new case for when the higher Herbrand difference is zero. Before doing this, we need to review some definitions. Let $(R,m,k)$ be a Noetherian local ring and let $M$ and $N$ be finitely generated modules. We write $\nu_{R}(M)$ for the minimal number of generators of $M$. The \textit{complexity} of the pair $(M,N)$, denoted by $\cx_{R}(M,N)$, is
\[\cx_{R}(M,N) = \inf\{b\in\mathbb{N} | \nu_{R}(\ext_{R}^{n}(M,N))\leq an^{b-1} \text{ for some real number }a\text{ and for all }n\gg 0\}.\]
The \textit{complexity} of $M$ is $\cx_{R}M=\cx_{R}(M,k)$ and the \textit{plexity} of $M$ is $\px_{R}M=\cx_{R}(k,M)$. Let
\[f_{tor}^{R}(M,N) = \inf\{s|\length(\tor_{i}^{R}(M,N))<\infty\text{ for all }i\geq s\},\]
and let
\[f_{ext}^{R}(M,N) = \inf\{s|\length(\ext_{R}^{i}(M,N))<\infty\text{ for all }i\geq s\}.\]
Let $e$ be an nonnegative integer. If $f_{tor}^{R}(M,N)<\infty$, we define $\eta_{e}^{R}(M,N)$ to be
\[\eta_{e}^{R}(M,N) = \lim_{n\to\infty} \frac{\sum_{i=f_{tor}^{R}(M,N)}^{n} (-1)^{i}\length(\tor_{i}^{R}(M,N))}{n^{e}}.\]
If $f_{ext}^{R}(M,N)<\infty$, we define the higher Herbrand difference $h_{e}^{R}(M,N)$ to be
\[h_{e}^{R}(M,N) = \lim_{n\to\infty} \frac{\sum_{i=f_{ext}^{R}(M,N)}^{n} (-1)^{i}\length(\ext_{R}^{i}(M,N))}{n^{e}}.\]
The invariant $\eta_{e}^{R}(M,N)$ was defined in \cite{dao3}, and it is a generalization of the Euler characteristic and Hochsters Theta Pairing (\cite{hochster}). It was used to study the asymptotic behavior of Tor. The invariant $h_{e}^{R}(M,N)$ was defined in \cite{celikbasdao}, and it is a generalization of the Euler Form and the Herbrand difference (\cite{buchweitz}). It was used to study the asymptotic behavior of Ext.\newline

The following is an Ext version of a theorem of Dao, and our proof mirrors his closely. In \cite[Theorem 6.3 and Corollary 8.3]{dao3} Dao showed that if $M$ and $N$ are finitely generated modules over a codimension $r$ local complete intersection $R$ with $f_{tor}^{R}(M,N)<\infty$ and $\eta_{r}^{R}(M,N)=0$ and if $n\geq 0$, then 
\[\tor_{n}^{R}(M,N)=\dots=\tor_{n+r-1}(M,N)=0\]
implies $\tor_{i}^{R}(M,N)=0$ for $i\geq n$. If $\length(M\otimes_{R}N)<\infty$ also holds, then 
\[\dim{M} + \dim{N} \leq \dim{R} + c - 1\]
implies $\eta_{c}^{R}(M,N)=0$, where $c=\max\{\cx_{R}M,\cx_{R}N\}$. 
\begin{theorem} \label{vanishingofextlemmausingherbranddifference} Let $R$ be a local complete intersection with infinite residue field, let $M$ and $N$ be finitely generated modules with $\length(M\otimes_{R}N)<\infty$, and let $c=\max\{\cx_{R}M,\cx_{R}N\}$. Assume
\[\dim{M} + \dim{N} \leq \dim{R} + c - 1.\]
Then $h_{c}^{R}(M,N)=0$. In particular, If
\[\ext_{R}^{n}(M,N)=\dots=\ext_{R}^{n+c-1}(M,N)=0\]
for some $n>\depth{R}-\depth{M}$, then $\ext_{R}^{i}(M,N)=0$ for all $i>\depth{R}-\depth{M}$.
\end{theorem}

\begin{proof}
If $c>\cx_{R}(M,N)$, then the result follows from \cite[Theorem 3.4.1 and Corollary 4.3]{celikbasdao}, and so we can assume $c=\cx_{R}(M,N)$. We can also assume $R$ is complete. Then by \cite[Theorem 1.3]{jorgensen2} there exists a complete intersection $Q$ and a $Q$ regular sequence $x_{1},\dots,x_{c}$ that satisfies the following properties
\begin{enumerate}
\item $R\cong Q/(x_{1},\dots,x_{c})$.
\item If $R_{j}=Q/(x_{1},\dots,x_{j})$, then $\cx_{R_{j}}M=\cx_{R}M-(c-j)$.
\item $\cx_{Q}M=0$ and $\cx_{Q}N=0$.
\end{enumerate}
Since $\cx_{Q}M=\cx_{Q}N=0$, we have $\pdim_{Q}M<\infty$ and $\pdim_{Q}N<\infty$. Also, our assumptions on dimension gives the following
\begin{align*}
\dim{M} + \dim{N} \leq \dim{R} + c - 1 = \dim{Q/(x_{1},\dots,x_{c})} + c - 1 &= \dim{Q}-c+c-1 \\
&= \dim{Q}-1 \\
&< \dim{Q}.
\end{align*}
Proposition \ref{xizerosummary} now gives $\xi^{Q}(M,N)=0$, and \cite[Theorem 3.4.3]{celikbasdao} gives the first equality below
\[2^{c}\cdot h_{c}^{R}(M,N) = h_{0}^{Q}(M,N) =  \xi^{Q}(M,N) = 0.\]
The second equality is by definition. Therefore the result follows from \cite[Theorem 4.2]{celikbasdao}.
\end{proof}

\begin{remark} A consequence of \cite[Corollary 3.5]{bergh} is that if $M$ and $N$ are finitely generated modules over a local complete intersection $R$ with $\length(N)<\infty$ and if 
\[\ext_{R}^{n}(M,N)=\dots=\ext_{R}^{n+\cx_{R}M-1}(M,N)=0\]
for some $n>\depth{R}-\depth{M}$, then $\ext_{R}^{i}(M,N)=0$ for all $i>\depth{R}-\depth{M}$ (also see \cite[Corollary 4.7]{celikbasdao}). Theorem \ref{vanishingofextlemmausingherbranddifference} does not assume either module has finite length, but instead assumes that a dimension inequality holds. 
\end{remark}



\small
\begin{thebibliography}{1}


\bibitem{avramovgasharovpeeva} Avramov, Luchezar L., Gasharov, Vesselin N., and Peeva, Irena V. {\em Complete Intersection Dimension}, Publications Mathématiques de l’IHÉS, vol. 86, no. 1, Dec. 1997, pp. 67–114. DOI.org (Crossref), https://doi.org/10.1007/BF02698901.

\bibitem{beder} Beder, Jesse. {\em The Grade Conjecture and Asymptotic Intersection Multiplicity}, Proceedings of the American Mathematical Society, vol. 142, no. 12, Aug. 2014, pp. 4065–77. DOI.org (Crossref), https://doi.org/10.1090/S0002-9939-2014-12183-6.

\bibitem{bergh} Bergh, Petter Andreas. {\em On the Vanishing of Homology with Modules of Finite Length}, Mathematica Scandinavica, vol. 112, no. 1, Mar. 2013, p. 11. DOI.org (Crossref), https://doi.org/10.7146/math.scand.a-15230.

\bibitem{brunsherzog} Bruns, Winfried, and Jurgen Herzog. {\em Cohen Macaulay Rings}, Revised Edition, Cambridge University Press, 1993.

\bibitem{buchweitz} Buchweitz, Ragnar-Olaf. {\em Maximal Cohen-Macaulay Modules and Tate Cohomology}, American Mathematical Society, 2021.

\bibitem{celikbasdao} Celikbas, Olgur, and Hailong Dao. {\em Asymptotic Behavior of Ext Functors for Modules of Finite Complete Intersection Dimension}, Mathematische Zeitschrift, vol. 269, no. 3–4, Dec. 2011, pp. 1005–20. DOI.org (Crossref), https://doi.org/10.1007/s00209-010-0771-9.

\bibitem{jeanchan} Chan, C. Y. Jean. {\em An Intersection Multiplicity in Terms of Ext-Modules}, Proceedings of the American Mathematical Society, vol. 130, no. 2, May 2001, pp. 327–36. DOI.org (Crossref), https://doi.org/10.1090/S0002-9939-01-06022-1.

\bibitem{dao3} Dao, Hailong. {\em Asymptotic behavior of Tor over complete intersections and applications}, https://arxiv.org/abs/0710.5818.

\bibitem{dao2} Dao, Hailong. {\em Some Observations on Local and Projective Hypersurfaces}, Mathematical Research Letters, vol. 15, no. 2, 2008, pp. 207–19. DOI.org (Crossref), https://doi.org/10.4310/MRL.2008.v15.n2.a1.

\bibitem{giletsoule} Gillet, H., and C. Soulé. {\em Intersection Theory Using Adams Operations}, Inventiones Mathematicae, vol. 90, no. 2, June 1987, pp. 243–77. DOI.org (Crossref), https://doi.org/10.1007/BF01388705. 

\bibitem{macaulay2} Grayson, Daniel R. and Stillman, Michael E. {\em Macaulay2, a software system for research in algebraic geometry}, http://www2.macaulay2.com.

\bibitem{heitmann} Heitmann, Raymond C. {\em A Counterexample to the Rigidity Conjecture for Rings}, Bulletin of the American Mathematical Society, vol. 29, no. 1, 1993, pp. 94–97. DOI.org (Crossref), https://doi.org/10.1090/S0273-0979-1993-00410-5.

\bibitem{hochster} Hochster, Melvin. {\em The Dimension of an Intersection in an Ambient Hypersurface}, Proceedings of the First Midwest Algebraic Geometry Conference, May 1980.

\bibitem{hochster2} Hochster, Melvin. {\em Euler Characteristics over Unramified Regular Local Rings}, Illinois Journal of Mathematics, vol. 28, no. 2, June 1984. DOI.org (Crossref), https://doi.org/10.1215/ijm/1256065276.

\bibitem{jorgensen2} Jorgensen, David A. {\em Complexity and Tor on a Complete Intersection}, Journal of Algebra, vol. 211, no. 2, Jan. 1999, pp. 578–98. DOI.org (Crossref), https://doi.org/10.1006/jabr.1998.7743.

\bibitem{jorgensen} Jorgensen, David A. {\em Finite Projective Dimension and the Vanishing of $\ext_{R}(M,M)$}, Communications in Algebra, vol. 36, no. 12, Dec. 2008, pp.4461-71. DOI.org (Crossref), https://doi.org/10.1080/00927870802179560.

\bibitem{jothilingam} Jothilingam, P. {\em A Note on Grade}, Nagoya Mathematical Journal, vol. 59, Dec. 1975, pp. 149–52. DOI.org (Crossref), https://doi.org/10.1017/S0027763000016858.

\bibitem{lichtenbaum} Lichtenbaum, Stephen. {\em On the Vanishing of Tor in Regular Local Rings}, Illinois Journal of Mathematics, vol. 10, no. 2, June 1966. DOI.org (Crossref), https://doi.org/10.1215/ijm/1256055103.

\bibitem{matsumura} Matsumura, Hideyuki. {\em Commutative Ring Theory}, Cambridge University Press, 1986.	

\bibitem{mori} Mori, Izuru. {\em Intersection Multiplicity over Noncommutative Algebras}, Journal of Algebra, vol. 252, no. 2, June 2002, pp. 241–57. DOI.org (Crossref), https://doi.org/10.1016/S0021-8693(02)00016-9.

\bibitem{mori3} Mori, Izuru. {\em Serre’s Vanishing Conjecture for Ext-Groups}, Journal of Pure and Applied Algebra, vol. 187, no. 1–3, Mar. 2004, pp. 207–40. DOI.org (Crossref), https://doi.org/10.1016/j.jpaa.2003.07.007.

\bibitem{mori2} Mori, Izuru, and S. Paul Smith. {\em Bézout’s Theorem for Non-Commutative Projective Spaces}, Journal of Pure and Applied Algebra, vol. 157, no. 2–3, Mar. 2001, pp. 279–99. DOI.org (Crossref), https://doi.org/10.1016/S0022-4049(00)00012-8.

\bibitem{roberts} Roberts, Paul. {\em Multiplicities and Chern Classes in Local Algebra}, Cambridge University Press, 1998.

\bibitem{roberts2} Roberts, Paul. {\em The Vanishing of Intersection Multiplicities of Perfect Complexes}, Bulletin of the American Mathematical Society, vol. 13, no. 2, Oct. 1985, pp. 127–31. DOI.org (Crossref), https://doi.org/10.1090/S0273-0979-1985-15394-7.

\bibitem{rotman} Rotman, Joseph J. {\em An Introduction to Homological Algebra}, 2nd ed, Springer, 2009.

\bibitem{serre} Serre, Jean-Pierre. {\em Local Algebra}, Springer Berlin Heidelberg, 2000. DOI.org (Crossref), https://doi.org/10.1007/978-3-662-04203-8.

\bibitem{soto} Soto Levins, Andrew J. {\em A Rigidity Theorem For Ext}, Journal of Commutative Algebra, vol. 16, no. 1, Mar. 2024. DOI.org (Crossref), https://doi.org/10.1216/jca.2024.16.115.

\end{thebibliography}
\end{document}